\documentclass[11pt, letterpaper]{article}

\newcommand{\scvx}{\texttt{SCvx}}

\newcommand{\longtitle}{Remarks on ``Successive Convexification: A Superlinearly Convergent Algorithm for Non-convex Optimal Control Problems''}

\usepackage{amsthm}                         % need for theorem-proof environment
\newtheoremstyle{plain}
  {}   % ABOVESPACE
  {}   % BELOWSPACE
  {\itshape}  % BODYFONT
  {}       % INDENT (empty value is the same as 0pt)
  {\mdseries\scshape} % HEADFONT
  {.}         % HEADPUNCT
  { } % HEADSPACE
  {\thmname{#1}\thmnumber{ #2}\ifx#3\empty\else\ (#3)\fi}
\theoremstyle{plain}
\newtheorem{theorem}{\underline{Theorem}}

\newtheorem{lemma}[theorem]{\underline{Lemma}}

\newtheoremstyle{definition}
  {}   % ABOVESPACE
  {}   % BELOWSPACE
  {}  % BODYFONT
  {}       % INDENT (empty value is the same as 0pt)
  {\mdseries\scshape} % HEADFONT
  {.}         % HEADPUNCT
  { } % HEADSPACE
  {\thmname{#1}\thmnumber{ #2}\ifx#3\empty\else\ (#3)\fi}
\theoremstyle{definition}

\newtheorem{remark}[theorem]{\underline{Remark}}
\newtheorem{assumption}[theorem]{\underline{Assumption}}
\usepackage[utf8]{inputenc}
\usepackage{lmodern}
\usepackage[T1]{fontenc}
\usepackage[numbers]{natbib}
\usepackage[margin=0.75in]{geometry}

\usepackage[dvipsnames]{xcolor}
\colorlet{RedOrange}{red!70!orange}
\colorlet{DullRed}{red!40!white!70!black}
\colorlet{PalePurple}{purple!30!blue!40!gray}

\usepackage{amsmath,amssymb,bm}

\usepackage{authblk}
\usepackage[hidelinks]{hyperref}

\newcommand{\behcet}{Beh\c{c}et}
\newcommand{\acikmese}{A\c{c}\i kme\c{s}e}

\title{\longtitle}
\author[1]{Dayou Luo\thanks{Email: dayoul@uw.edu}}
\author[2]{Purnanand Elango\thanks{Email: pelango@uw.edu}}
\author[2]{\behcet \acikmese \thanks{Email: behcet@uw.edu}}

\affil[1]{Department of Applied Mathematics}
\affil[2]{William E. Boeing Department of Aeronautics and Astronautics}
\affil[ ]{University of Washington, Seattle, WA, 98195}
\date{}

\begin{document}

\maketitle

\begin{abstract}
The purpose of this note is to highlight and address inaccuracies in the convergence guarantees of {\scvx}, a nonconvex trajectory optimization algorithm proposed in \citet{mao2018scvx}, and make connections to relevant prior work. Specifically, we identify errors in the convergence proof within \citet{mao2018scvx} and reestablish the proof of convergence by employing a new method under stricter assumptions. %Additionally, we point out and address other inaccuracies in \cite{mao2018scvx}.
\end{abstract}

\defcitealias{mao2018scvx}{SCvx} 

\section{Introduction}
This note identifies and remedies mistakes in \citet{mao2018scvx}. A sequential convex programming (SCP) algorithm, called {\scvx}, for solving discrete-time nonlinear optimal control problems is proposed in \citet{mao2018scvx}, which first converts the original constrained optimization problem into an unconstrained problem through the use of exact penalty \cite{han1979exact}, and then computes a stationary point of the unconstrained problem via a trust-region method. Unlike typical trust-region methods \cite[Chapter 4]{nocedal1999numerical}, the proposed approach does not utilize any second-order information such as the Hessian of the Lagrangian; it relies solely on first-order information, which can potentially improve the overall computational effort. %

\citet{mao2018scvx} demonstrated that {\scvx} ensures stationarity for unconstrained problems at accumulation points of the iterates, referred to as weak convergence. \citet{mao2018scvx} also showed that {\scvx} ensures strong convergence, meaning the sequence of iterates will converge to a single point. Such a convergence achieves a superlinear rate under specific conditions.

There are close connections between {\scvx} and prior work on trust-region methods from the 1980s. \citet{zhang1985improved} proposed an algorithm similar to {\scvx} and proved its global convergence, while \citet{zhang1989superlinear} demonstrated that, given strong convergence, the algorithm exhibited superlinear convergence rate under certain conditions. The main contribution of \citet{mao2018scvx} is the proof of strong convergence, which serves as the link between the results established in \cite{zhang1985improved} and \cite{zhang1989superlinear}.

The authors of this note, however, identified errors in the proof of strong convergence of {\scvx} that are challenging to rectify. This note provides a proof of strong convergence for {\scvx}, under the assumptions made by \citet{zhang1989superlinear}, which are more stringent than the original assumptions made by \citet{mao2018scvx}. Furthermore, invoking the assumptions of \cite{zhang1989superlinear} also provides the superlinear convergence rate for {\scvx}.

Given the relatively stringent assumptions needed for the results of this note, future work can explore strong convergence of {\scvx} under weaker conditions. Additionally, recognizing that the assumptions needed for superlinear convergence rate are uncommon in real-world applications, future work could be on the characterization of the convergence rate under weaker, more practical assumptions.

\section{Errors and Corrections in \texorpdfstring{\citet{mao2018scvx}}{citemaoscvx}}

In this section, we enumerate the errors and ambiguities in \citet{mao2018scvx}. We have categorized these errors into two types: those that have a straightforward remedy, and those that do not. The following list shows the first type of errors along with their corresponding remedies. Note that the page numbers referenced and the literature cited correspond to those in the appended text of \citet{mao2018scvx}.

\begin{itemize}
    \item Definition (3.10), on page 11, is an incorrect definition for the general subdifferential $\partial J(\bar z)$. The correct definition should be:
    \[
    \partial J(\bar z) = \{\eta \mid \eta^\top s \leq d J(\bar{z}, s) \text{ for all } s\}.
    \]
    \item Theorem 3.9, on page 12, makes an inaccurate reference to the exact penalty result in [24] by assuming $\bar{z}$ to be a KKT point. The accurate version should replace the term ``KKT point'' by ``local minimizer'' in the first part of the statement.
   
    \item In equation (3.19), on page 14, the second inequality is incorrect because of the ambiguity in the sign of $o(r)$. Taking the absolute value of $o(r)$ should rectify this error.
    
    \item The reference to the Bolzano-Weierstrass Theorem in Theorem 3.13, on page 14, is incorrect. By the definitions of $J(z)$ and $L(d)$, the domains of both functions are not compact sets. A fix for this is to assume that the objective function $J$ has compact level sets. Assumption \ref{assumption: compact} and Lemma \ref{lemma: compact} will rectify this. 

    \item On page 21, the statement \textbf{``A particular advantage of this approach is that it always produces a strictly complementary solution.''} is misleading. The strictly complementarity property is unrelated to the solution method for the convex subproblems via an interior-point method.
     
    \item In Theorem 4.7, on page 25, the statement \textbf{``Given a sequence \(\{z^k\}\) weakly converging to  \(\bar{z}\)''} is incorrect. The proof of Theorem 4.7 requires the convergence of the entire sequence \(z^k\) to \(\bar{z}\), rather than just the convergence of a subsequence. We can correct this by replacing the weak convergence with strong convergence. 
\end{itemize}

Next, we list the more serious errors.
\begin{itemize}
    \item Equation (3.33), on page 17, cannot result from (3.15) and (3.18) in Lemma 3.11.
    \item On page 26, the statement \textbf{``From (3.19) in the proof of Lemma 3.11, $\rho^k \rightarrow 1$...''} is incorrect.
\end{itemize}

Both errors come from misinterpretations of the results in Lemma 3.11. The premise for the validity of the conclusions in Lemma 3.11 is that the iterates are confined to a neighborhood of a nonstationary point. However, the iterates generated by {\scvx}, upon convergence, approach a stationary point. Therefore, Lemma 3.11 does not apply to the proofs of both Condition 3.17 and Theorem 4.7.

The second error can be remedied: we refer the reader to \cite[Theorem 2.1]{zhang1989superlinear} for a resolution. However, the first error cannot be rectified easily, causing the entire Section 3.2 on strong convergence guarantee to be incorrect. In the subsequent section of this note, we will use a new method to prove strong convergence. This method, however, requires tighter assumptions. Thus, we do not provide a complete fix for the proof of strong convergence presented in \citet[Section 3.2]{mao2018scvx}.

\section{Proof of Strong Convergence}
We follow the same notation as presented in Section 4 of \citet{mao2018scvx}. The objective function can be expressed as:
\begin{equation}
    J(z) = \psi(G(z))
\end{equation}
where \(\psi\) represents a convex function and \(G\) is a $C^1$ smooth function. Let $z^k$ be a sequence generated by {\scvx} and $z^0$ be the initial item of the sequence $z^k$. To ensure that the iterations \(z^k\) remain bounded, we impose the following assumption:
\begin{assumption}
\label{assumption: compact}
 The set \(\{z \mid J(z) \leq J(z^0)\}\) is compact.   
\end{assumption}

\begin{lemma}
\label{lemma: compact}
Given that $\{z \mid J(z) \leq J(z^0)\}$ is compact, iteration $z^k$ is restricted to this compact set and $\psi, G$, and $\nabla G$ are Lipschitz continuous on this compact set.
\end{lemma}

\begin{proof}
This result is straightforward, as $J(z^k)$ is a decreasing sequence due to the design of the {\scvx} algorithm.
\end{proof}

Since \( z^k \) is bounded, a subsequence of \( z^k \) will, by compactness, converge to a point \( \bar{z} \). The aim of this section is to prove that \( \lim_{k \to \infty} z^k = \bar{z} \) for the whole sequence. To achieve this, we rely on the results from Lemma 4.5 and Lemma 4.6 in \cite{mao2018scvx}.

Suppose the assumptions in Lemma 4.5 are satisfied at \(\bar 
z\), Lemma 4.5 guarantees that there exist $\beta>0$ and $\delta>0$ such that, for all $z \in N(\bar{z}, \delta):=\{z \mid\|z-\bar{z}\| \leq \delta\}$, we have:
\begin{equation}
\label{equ: local J}
J(z)-J(\bar{z}) \geq \beta\|z-\bar{z}\|,
\end{equation} 

Under the same assumption in Lemma 4.5, Lemma 4.6 indicates that there exists $\gamma > 0$, such that:
\begin{equation*}
\psi\left(G(\bar{z}) + \nabla G(\bar{z})^\top d\right) \geq \psi(G(\bar{z})) + \gamma\|d\|, \quad \forall d \in \mathbb{R}^{n_z}.
\end{equation*}

Notice that \eqref{equ: local J} is only valid in a neighborhood of \( \bar{z} \). We aim to show that when \( z^k \) is sufficiently close to \( \bar{z} \), \( z^{k+1} \) will be within the neighborhood that validates \eqref{equ: local J}. The following theorem demonstrates that the update \( d_k \) computed by solving the subproblem at \( z^k \) will be sufficiently small, regardless of the trust-region radius. We will later use this to prove the validity of \eqref{equ: local J} for \( z^{k+1} \).

\begin{theorem}
\label{thm: minimum trust region solution}
Suppose $\bar{z}$ is a feasible stationary point that follows Assumption 4.1 and Assumption 4.2. For any $\epsilon > 0$, there exists $\eta > 0$ such that for any $z^k$ with $\|z^k - \bar{z}\| < \eta$, the solution of the trust-region subproblem ${d}^k$ will satisfy $\|{d}^k\| < \epsilon$.
\end{theorem}

\begin{proof}
Consider the following equation:
\begin{align}
\psi(G(z^k) + \nabla G(z^k) d) - \psi(G(z^k)) ={} & \psi(G(z^k) + \nabla G(z^k) d) - \psi(G(\bar{z}) + \nabla G(\bar{z}) d) + \psi(G(\bar{z}) + \nabla G(\bar{z}) d) \label{equ: 1}\\
& - \psi(G(\bar{z})) + \psi(G(\bar{z})) - \psi(G(z^k)).\nonumber         
\end{align}
By Lemma 4.6, we have:
\begin{equation}
\label{equ: 2}
\psi\left(G\left(\bar{z}\right) + \nabla G\left(\bar{z}\right) d\right) - \psi\left(G\left(\bar{z}\right)\right) \geq \gamma\|d\| \text{ for all } d.
\end{equation}

By the Lipschitz continuity of $\psi, G$, and $\nabla G$ on the compact set \(\{z \mid J(z) \leq J(z^0)\}\), there exists $\eta$, such that for all $\|z^k - \bar{z}\| < \eta$, we have:
\begin{equation}
\label{equ: 3}
\begin{aligned}
& \|\psi(G(z^k) + \nabla G(z^k) d) - \psi(G(\bar{z}) + \nabla G(\bar{z}) d) + \psi(G(\bar{z})) - \psi(G(z^k))\| \leq \frac{\gamma}{2} \epsilon \quad \text{for any } \|d\| = \epsilon.
\end{aligned}
\end{equation}

Plugging \eqref{equ: 2} and \eqref{equ: 3} into \eqref{equ: 1}, we have:
\begin{equation}
\psi\left(G\left(z^k\right) + \nabla G\left(z^k\right) d\right) - \psi\left(G\left(z^k\right)\right) \geq \frac{\gamma}{2} \epsilon  \quad \text{for any } \|d\| = \epsilon.
\end{equation}

Since $\psi\left(G\left(z^k\right) + \nabla G\left(z^k\right) d\right)$ is a convex function of $d$, the minimizer of this function, $\bar{d}^k$, will satisfy that $\|\bar{d}^k\| \leq \epsilon$. Thus, the solution $d^k$ of the constrained trust-region subproblem will have:
\[
\|d^k\| \leq \|\bar{d}^k\| \leq \epsilon.\]\end{proof}

With Theorem \ref{thm: minimum trust region solution}, we are now able to show that \( z^{k+1} \) is not far from \( z^k \) when the latter is close to \( \bar{z} \). Consequently, \eqref{equ: local J} will be proven to be valid at \( z^{k+1} \), and we will have:
\[
\|z^{k+1} - \bar{z}\| \leq \frac{J(z^{k+1}) - J(\bar{z})}{\gamma}. 
\]
The above expression provides an upper bound on \( \|z^{k+1} - \bar{z}\| \). This allows us to demonstrate that \( z^{k+1} \) also satisfies the assumption in Theorem \ref{thm: minimum trust region solution}. Meanwhile, the distance to \( \bar{z} \) is bounded as the function \( J(z^k) \) decreases over iterations. We will then use induction to establish the convergence of \( z^k \).

\begin{theorem}
    Let $\bar{z}$ be a limit point of the sequence $z^k$. Suppose that $\bar{z}$ satisfies the assumptions required in Lemma 4.5 and $J(\bar{z}) > -\infty$. Additionally, suppose that Assumption \ref{assumption: compact} holds. Then $\lim_{k \to \infty} z^k = \bar{z}$.
\end{theorem}

\begin{proof}
Theorem 3.13, on page 14, ensures that $\bar z$ is a stationary point of $J(\cdot).$

Under Assumptions 3.5, 4.1, and 4.2, Lemma 4.5 ensures that there exists $\gamma > 0$ and $\delta > 0$ such that, for all $z \in N(\bar{z}, \delta) := \{z \mid \|z - \bar{z}\| \leq \delta\}$, we have:
\begin{equation*}
J(z) - J(\bar{z}) \geq \gamma\|z - \bar{z}\|.
\end{equation*}
Meanwhile, due to the continuity of $\psi, G,$ and $\nabla G$, $J$ is Lipschitz continuous on the compact set \(\{z \mid J(z) \leq J(z^0)\}\). Furthermore, since $J(z^k)$ is a decreasing sequence of $k$, we have $J(\bar z) = \inf_k J(z^k).$

According to Theorem \ref{thm: minimum trust region solution}, for $\epsilon = \delta/2 > 0$, there exists $\eta > 0$ such that for any $z^k$ with $\|z^k - \bar{z}\| < \eta$, the solution of the constrained trust-region subproblem, $d^k$, will satisfy $\|d^k\| < \delta/2$.

Next, we use mathematical induction to show that $\lim_{k \to \infty} z^k = \bar{z}$.

Since $\bar{z}$ is an accumulation point of the iterates of {\scvx}, for arbitrarily small $\eta_1 < \eta$, there exists a $k \in \mathbb{N}$ satisfying: 
\begin{equation}
\|z^k - \bar{z}\| \leq \min(\eta_1, \delta/2),
\end{equation}
and: 
\begin{equation}
\frac{J(z^k) - J(\bar{z})}{\gamma} \leq \min(\eta_1, \delta/2).
\end{equation}

The first inequality holds because $\bar{z}$ is an accumulation point. The second inequality is guaranteed as $J$ is a continuous function. 

Now, we prove that for $z^{k+1}$ satisfying:
\begin{equation}
\|z^{k+1} - \bar{z}\| \leq \min(\eta_1, \delta/2)
\end{equation}
and:
\begin{equation}
\label{equ: J decreasing new}
\frac{J(z^{k+1}) - J(\bar{z})}{\gamma} \leq \min(\eta_1, \delta/2).
\end{equation}

First, since the trust-region ratio $\rho^k0$ is positive, we have:
\begin{equation}
J(\bar{z}) \leq J(z^{k+1}) < J(z^k).
\end{equation}

Therefore, \eqref{equ: J decreasing new} is automatically true. 

Next, we observe that:
\begin{equation}
\|z^{k+1} - \bar{z}\| < \|z^k - \bar{z}\| + \|z^{k+1} - z^k\| = \|z^k - \bar{z}\| + \|d^k\| \leq \delta/2 + \delta/2  = \delta.
\end{equation}
So, Lemma 4.5 is valid, and thus:
\begin{equation*}
J(z^{k+1}) - J(\bar{z}) \geq \gamma\|z^{k+1} - \bar{z}\|.
\end{equation*}
Thus \eqref{equ: J decreasing new} indicates that:
\begin{equation}
\|z^{k+1} - \bar{z}\| \leq \frac{J(z^{k+1}) - J(\bar{z})}{\gamma} \leq \min(\eta_1, \delta/2).
\end{equation}
By induction, for all $n \in \mathbb{N}$, we have:
\begin{equation}
\|z^{k+n} - \bar{z}\| \leq \min(\eta_1, \delta/2)
\end{equation}
and:
\begin{equation}
\frac{J(z^{k+n}) - J(\bar{z})}{\gamma} \leq \min(\eta_1, \delta/2).
\end{equation}

Since we are able to pick $\eta_1$ arbitrarily small, we have $\lim_{k \to \infty} z^k = \bar{z}$.\end{proof}
\begin{remark}
We want to highlight that Assumption 4.2 in \citet{mao2018scvx} is rather strict. This is due to the binding requirement that 
\[ \left|I_{ac}(\bar{z})\right| + \left|J_{ac}(\bar{z})\right| \geq n_u(N-1), \]
and when connected with LICQ, we need this inequality to become an equality. It implies that we must activate exactly \( n_u(N-1) \) inequalities, which is often difficult to accomplish. Even when the Bang-Bang principle is applicable, the presence of discretization means that we cannot ensure that exactly \( n_u(N-1) \) inequalities will be activated.

\end{remark}
\bibliographystyle{plainnat}
\bibliography{reference}

\begin{thebibliography}{5}
\providecommand{\natexlab}[1]{#1}
\providecommand{\url}[1]{\texttt{#1}}
\expandafter\ifx\csname urlstyle\endcsname\relax
  \providecommand{\doi}[1]{doi: #1}\else
  \providecommand{\doi}{doi: \begingroup \urlstyle{rm}\Url}\fi

\bibitem[Han and Mangasarian(1979)]{han1979exact}
Shih-Ping Han and Olvi~L Mangasarian.
\newblock Exact penalty functions in nonlinear programming.
\newblock \emph{Mathematical Programming}, 17:\penalty0 251--269, 1979.
\newblock URL \url{https://doi.org/10.1007/BF01588250}.

\bibitem[Mao et~al.(2018)Mao, Szmuk, Xu, and Acikmese]{mao2018scvx}
Yuanqi Mao, Michael Szmuk, Xiangru Xu, and Behcet Acikmese.
\newblock Successive convexification: A superlinearly convergent algorithm for non-convex optimal control problems.
\newblock \emph{arXiv:1804.06539 version 2}, 2018.
\newblock URL \url{https://doi.org/10.48550/arXiv.1804.06539}.

\bibitem[Nocedal and Wright(1999)]{nocedal1999numerical}
Jorge Nocedal and Stephen~J Wright.
\newblock \emph{Numerical Optimization}.
\newblock Springer, 1999.
\newblock URL \url{https://doi.org/10.1007/978-0-387-40065-5}.

\bibitem[Zhang(1989)]{zhang1989superlinear}
Jianzhong Zhang.
\newblock Superlinear convergence of a trust region-type successive linear programming method.
\newblock \emph{Journal of Optimization Theory and Applications}, 61\penalty0 (2):\penalty0 295--310, 1989.
\newblock URL \url{https://doi.org/10.1007/BF00962802}.

\bibitem[Zhang et~al.(1985)Zhang, Kim, and Lasdon]{zhang1985improved}
Jianzhong Zhang, Nae-Heon Kim, and L~Lasdon.
\newblock An improved successive linear programming algorithm.
\newblock \emph{Management Science}, 31\penalty0 (10):\penalty0 1312--1331, 1985.
\newblock URL \url{https://doi.org/10.1287/mnsc.31.10.1312}.

\end{thebibliography}

\end{document}